\documentclass[11pt, reqno]{amsart}
\usepackage{amsmath}
\usepackage{amssymb}
\usepackage{esint, enumitem}
\usepackage[hidelinks]{hyperref}
\usepackage{color}
\newtheorem{theorem}{Theorem}[section]
\newtheorem{lemma}[theorem]{Lemma}
\newtheorem{proposition}[theorem]{Proposition}

\newtheorem{problem}[theorem]{Problem}

\theoremstyle{definition}

\theoremstyle{remark}
\newtheorem{remark}[theorem]{Remark}

\numberwithin{equation}{section}

\newcommand{\bb}[1]{\mathbb{#1}}
\newcommand{\dd}{\mathrm{d}}

\allowdisplaybreaks[4]

\title[Geometric regularity for the special Lagrangian equation]{Geometric regularity for semi-convex viscosity solutions to the special Lagrangian equation}

\date{\today}

\author{Zhenyu Fan}
\address{School of Mathematical Sciences, Peking University, Beijing, 100871, P. R. China}
\email{fanzhenyu@stu.pku.edu.cn}

\begin{document}

\subjclass[2010]{
                35B65;\,%Smoothness and regularity of solutions to PDEs
                35J60;\,%Nonlinear elliptic equations
                53D12. %Lagrangian submanifolds; Maslov index
                }
\keywords{Regularity, viscosity solutions, the special Lagrangian equation.}

\begin{abstract}
   We prove that the gradient graph of any $W^{2,1}$ semi-convex viscosity solution to the three-dimensional subcritical special Lagrangian equation is area-minimizing and smooth. Furthermore, we prove the smoothness of $W^{2,1}\cap C^{1,1/3+}$ semi-convex viscosity solutions. This yields progress on the conjecture proposed by Wang--Yuan \cite{Wang-Yuan-singular}. The corresponding results in higher dimensions are also established. 
\end{abstract}

\maketitle

\section{Introduction.} 

In this article, we study the regularity of semi-convex viscosity solutions to the special Lagrangian equation
\begin{equation}\label{eq: sLag}
    F(D^2u)=\sum_{i=1}^{n} \arctan\lambda_i(D^2u)=\Theta,
\end{equation}
where $\lambda_i$ are the eigenvalues of $D^2u$ and the phase $\Theta\in (-n\pi/2, n\pi/2)$ is a constant. The fully nonlinear equation above arises from the calibrated geometry of Harvey and Lawson. For a $C^2$ potential $u$ defined on a domain in $\bb{R}^n$, its Lagrangian graph $\Gamma_u=(x,Du(x)) \subset \bb{R}^n \times \bb{R}^n$ is called special if the phase, defined pointwise as the argument of the complex number $(1+\sqrt{-1}\lambda_1)\cdots(1+\sqrt{-1}\lambda_n)$, equals some constant $\Theta$, that is, $u$ satisfies equation \eqref{eq: sLag}. In this case, $\Gamma_u$ is an area-minimizing submanifold in $(\bb{R}^n\times \bb{R}^n , \dd x^2+\dd y^2)$.

The regularity theory for the equation \eqref{eq: sLag} strongly depends on the value of $\Theta$. Yuan \cite{Yuan-06} observed that the level set $\{F(M)=\Theta\}$ is convex if and only if $|\Theta|\ge (n-2)\pi/2$, thus the phase is called \emph{critical} when$|\Theta|=(n-2)\pi/2$, \emph{supercritical} when $|\Theta|>(n-2)\pi/2$ and \emph{subcritical} when $|\Theta|<(n-2)\pi/2$.

For critical and supercritical phases, the regularity theory for viscosity solutions of equation \eqref{eq: sLag} is now well developed. The \emph{a priori} gradient and Hessian estimates were established by Warren--Yuan \cite{Warren-Yuan-grad} and Wang--Yuan \cite{Wang-Yuan-Hess}, respectively. As a result, viscosity solutions to \eqref{eq: sLag} are known to be smooth. This follows from smooth approximation using \cite{CNS-3}, together with the Evans--Krylov--Safonov theory. For variable phases, such \emph{a priori} estimates have also been obtained in the following works \cite{AB-21,AB-24,AB-Mooney-Shankar, ZhouXC, DingQ-Lag, AB-RS-JW}.

The situation becomes much worse in the subcritical range. In dimension three, Nadirashvili and Vl\u{a}du\c{t} \cite{NV-sLag} constructed singular viscosity solutions of class $C^{1,1/3}$ for arbitrary subcritical phases. More generally, for any $m \geq 2$, Wang--Yuan \cite{Wang-Yuan-singular} constructed singular viscosity solutions of class $C^{1,1/(2m-1)}$. It is important to point out that the singularities in these examples are not geometrical. In other words, their gradient graphs are smooth area-minimizing submanifolds as geometric objects. These examples are obtained by starting from a smooth area-minimizing special Lagrangian graph and then rotating it so that vertical tangent directions appear at the origin, thereby producing a singularity for the potential. Since all these examples are at least $C^1$, and none of them has regularity beyond $C^{1,1/3}$, Nadirashvili--Vl\u{a}du\c{t} and Wang--Yuan proposed the following two conjectures:

\begin{problem}[\cite{NV-sLag}]\label{PROB: NV}
Is every viscosity solution to \eqref{eq: sLag} of class $C^1$?
\end{problem}
\begin{problem}[ \cite{Wang-Yuan-singular}]\label{PROB: Wang-Yuan}
    Is every $C^{1, 1/3+}$ viscosity solution to \eqref{eq: sLag} smooth? Can this at least be proved in dimension three?
\end{problem}
There is a clue for Problem \ref{PROB: Wang-Yuan} from the Monge-Amp\`ere equation. Pogorelov \cite{Pogorelov} constructed singular solutions of class $C^{1,1-2/n}$ to $\det D^2u=1$. Later, Caffarelli \cite{Caffarelli-90} and Urbas \cite{Urbas-reg} proved smoothness of $C^{1,1-2/n +}$ viscosity solutions. In dimension three, this threshold is precisely $C^{1,1/3}$.

Problem \ref{PROB: NV} has recently been answered negatively by Mooney and Savin \cite{Mooney-Savin}. They constructed a Lipschitz, semi-convex viscosity solution at a subcritical phase in dimension three. More importantly, the gradient graph of this singular solution is neither smooth nor minimal. Here, at non-differentiable points, the gradient is understood in the sense of the subgradient of a semi-convex function. This shows that the most important geometric property of smooth solutions to equation \eqref{eq: sLag}, namely the minimality of the gradient graph, does not in general pass to viscosity solutions. Later, Mooney and Shankar \cite{Mooney-Shankar} extended the counterexample of Mooney--Savin to arbitrary subcritical phases in all dimensions. From the construction of Mooney--Savin, one can see that the gradient graph of their Lipschitz solution contains a “vertical piece” of positive measure. This is the key obstruction to minimality. This naturally leads to the following question:
\begin{problem}\label{PROB: geo reg}
    Under what conditions on a viscosity solution to \eqref{eq: sLag} can one ensure that its gradient graph is a  smooth area-minimizing Lagrangian submanifold in $\bb{R}^n\times \bb{R}^n$. We call this the geometric regularity problem.
\end{problem}
We expect that a natural condition for minimality may be $C^1$, since the gradient graph of a $C^1$ function has no ``vertical piece''. However, it may still contain points with vertical tangent directions, and the set of such points may have positive measure. Therefore, it is not clear whether the $C^1$ condition is sufficient. From another point of view, by the calibration argument \cite[Theorem 4.2]{Harvey-Lawson}, it is known that the gradient graph of a $W^{2,n}$ strong solution is area-minimizing. It is therefore reasonable to expect that some $W^{2,p}$ condition with $p<n$ should be sufficient.

In this article, we prove several results related to Problems \ref{PROB: Wang-Yuan} and \ref{PROB: geo reg} for semi-convex viscosity solutions. In dimension three, our main result is as follows.

\begin{theorem}\label{thm: 3D main}
    Let $u$ be a $W^{2,1}$ semi-convex viscosity solution to \eqref{eq: sLag} on $B_1\subset \bb{R}^3$, then its gradient graph is a smooth area-minimizing Lagrangian submanifold in $\bb{R}^3\times \bb{R}^3$. If, in addition, $u\in C^{1,\alpha}(B_1)$ for some $\alpha>1/3$, then $u$ is smooth in $B_1$.
\end{theorem}
In higher dimensions, we divide the subcritical range into $n-2$ intervals of length $\pi$:
\begin{align*}
    &I_{k}= \left( -(n-2)\frac{\pi}{2}+(k-1)\pi, -(n-2)\frac{\pi}{2}+k\pi \right] \quad \text{for}\ k=1,\dots,n-3,\\
    &\text{and}\qquad I_{n-2}= \left((n-4)\dfrac{\pi}{2}, (n-2)\dfrac{\pi}{2}\right).
\end{align*}
The corresponding higher-dimensional result  is stated as follows.

\begin{theorem}[Minimality of the gradient graph]\label{THM: minimality in high dim}
    Let $u$ be a semi-convex viscosity solution to \eqref{eq: sLag} with phase $\Theta$ on $B_1\subset\bb{R}^n$. Suppose $\Theta\in I_k$ for some $k\in\{1,2,\dots, n-2\}$ and $u\in W^{2,k}(B_1)$. Then its gradient graph is an area-minimizing Lagrangian submanifold in $\bb{R}^n\times \mathbb{R}^n$.
\end{theorem}

\begin{theorem}\label{THM: high dim reg}
    Under the hypotheses of Theorem \ref{THM: minimality in high dim}, assume further that $u$ satisfies the following semi-convexity condition:
    
    $\bullet$ $u$ is generally semi-convex when $n=3$ or $4$, i.e., $u+ \tan \theta |x|^2/2$ is convex for some fixed angle $ \theta \in (0,\pi/2)$;

    $\bullet$ $u$ is $\tan(\pi/6)$-convex when $n\ge 5$, i.e., $u+\tan(\pi/6)|x|^2/2$ is convex.

    Then its gradient graph is a smooth area-minimizing Lagrangian submanifold in $\bb{R}^n\times\bb{R}^n$. If, in addition, $u\in C^{1,\alpha}(B_1)$ for some $\alpha>1/3$, then $u$ is smooth in $B_1$.
\end{theorem}

\begin{remark}
    It is natural to consider semi-convex solutions, since all the singular solutions mentioned above are semi-convex. Thus, general semi-convexity alone does not give any regularity better than Lipschitz. 
    
    We also note that regularity can be established for convex solutions, or for solutions satisfying a certain quantitative, phase-dependent semi-convexity condition; we refer to Chen--Shankar--Yuan \cite{Chen-Shankar-Yuan} and Mooney--Shankar \cite{Mooney-Shankar}. Moreover, the examples of Mooney--Shankar show that their quantitative, phase-dependent semi-convexity condition is sharp. Therefore, additional assumptions are necessary if one wants to obtain regularity results for general semi-convex solutions.
\end{remark}

\begin{remark}
    Since the examples of Mooney--Savin and Mooney--Shankar have non-smooth and non-minimal gradient graphs, our three-dimensional result shows that these examples cannot belong to $ W^{2,1}$. In particular, $W^{2,1}$ regularity fails in general for viscosity solutions of the subcritical special Lagrangian equation. A recent example constructed by the author, Li and Wang \cite{Fan-Li-Wang} also shows that \emph{a priori} $W^{2,1}$ estimates fail for the subcritical case. We also remark that the $C^{1/(2m-1)}$ singular solutions constructed by Wang--Yuan belong to $W^{2,p}$ for any $p<(2m+1)/(2m-2)$ \cite[Theorem 1.3]{Wang-Yuan-singular}.
\end{remark}

 We now explain the main ideas of the proof. We focus on the three-dimensional case. Our proof follows the framework of Chen--Shankar--Yuan \cite{Chen-Shankar-Yuan}. They proved that the Lewy--Yuan rotation $(x,y)\mapsto (x\cos\beta +y\sin\beta  ,-x\sin\beta+y\cos\beta)$ is valid for general semi-convex functions $u$. After this rotation, the gradient graph $(x,\partial u(x))$ has bounded slope and can be written as the gradient graph $(\bar x, D\bar u (\bar x))$ of a new potential $\bar u\in C^{1,1}$ in the rotated $\bar x$-coordinates. Therefore, to prove the geometric regularity of the gradient graph, it suffices to prove the regularity of the rotated potential $\bar u$.

 A key question is whether the rotated potential $\bar u$ still satisfies equation \eqref{eq: sLag} with the corresponding phase, at least in the viscosity sense. If this is true, then the $C^{1,1}$ regularity of $\bar u$ directly implies the minimality of the gradient graph $(\bar x,D\bar u(\bar x))$. The smoothness of the gradient graph follows from the rigidity of three-dimensional graphical minimal cones \cite{Fischer-Colbrie, Han-Nadi-Yuan} and the minimal surface theory.

 Chen--Shankar--Yuan proved that supersolutions are preserved under the rotation for arbitrary semi-convex solutions. However, the preservation of subsolutions is more subtle. They proved it for convex solutions by mollification, and the concavity of equation \eqref{eq: sLag} on the convex branch is used. Under the phase-dependent semi-convexity assumption of Mooney--Shankar, one can directly verify the preservation of subsolutions when some eigenvalues of $D^2u$ blow up. For general semi-convex solutions, however, this preservation may fail, as shown by the examples of Mooney--Savin and Mooney--Shankar whose gradient graphs are non-minimal.

 In our setting, with the additional $W^{2,1}$ assumption, we do not try to verify the preservation of subsolutions directly. Instead, we show that the rotated potential satisfies the special Lagrangian equation with the corresponding phase almost everywhere. Let $U(x)= \sin\beta \, u(x)+ \cos\beta\, |x|^2/2$ be the uniformly convex potential associated with the Lewy--Yuan rotation, so that the coordinate change is given by $ \bar x=DU(x) $, with inverse $x=DU^*(\bar x)$. In the rotated coordinates, let $\mathcal{R}$ denote the \emph{regular set},  where $D^2U^*$ is nondegenerate; and let $\mathcal{S}$ denote the \emph{singular set}, where $D^2U^*$ is degenerate. At points of $\mathcal{R}$, one can directly check, using the coordinate change , that the rotated potential $\bar u$ and satisfies the corresponding equation. On the other hand, by combining Alexandrov's theorem with Savin's small perturbation theorem \cite{Savin}, we know that $u$ is smooth near almost every point in the original $x$-coordinates. It follows that the preimage $DU^*(\mathcal{S})$ has measure zero. Thus, the main issue is to show that $\mathcal{S}$ itself has measure zero in the rotated coordinates. This is where the $W^{2,1}$ assumption used.  Since $U\in W^{2,1}$, the measure $\Delta U\,\dd x$ has no mass on  the null set $DU^*(\mathcal{S})$. Heuristically, the change of variables gives the duality $\Delta U(x)\,\dd x= \sigma_2(D^2U^*(\bar x)) \,\dd \bar x $. Hence,
    \[  0 = \int_{DU^*(\mathcal{S})}\Delta U(x)\,\dd x= \int_{\mathcal{S}}\sigma_2(D^2U^*(\bar x))\,\dd \bar x.   \]
 If $\mathcal{S}$ has positive measure, this implies that $D^2U^*$ has at least two degenerate eigenvalues in $\mathcal{S}$. Geometrically, these zero eigenvalues correspond, under the Lewy--Yuan rotation, to blow-up eigenvalues of $D^2u$ in the original coordinates. Thus, $D^2u$ would have to blow up in at least two directions at the corresponding point. However, at the subcritical phase, $D^2u$ has at most one blow up eigenvalue. This contradiction shows that $\mathcal{S}$ has measure zero, and hence the geometric regularity follows.

 Finally, we need to pass the regularity of the rotated potential $\bar u$ back to the original solution $u$. In the works of Chen--Shankar--Yuan \cite{Chen-Shankar-Yuan}, Bhattacharya--Shankar \cite{AB-RS-Crelle, AB-RS-ARMA}, and Mooney--Shankar \cite{Mooney-Shankar}, this is usually done by using the strong maximum principle or  constant rank theorems. However, these arguments rely on the convexity or inverse-convexity of the special Lagrangian operator. In our setting, such a structure is not available.
 For this reason, we impose the additional assumption $u\in C^{1,1/3+}$. Then, using the argument of Bhattacharya--Shankar \cite[Theorem 4.1]{AB-RS-ARMA}, we compare the growth of the potential in the coordinates before and after the rotation. This allows us to prove that the coordinate change maps of the Lewy--Yuan rotation are nondegenerate. As a result, the regularity of the rotated potential can be transferred back to the original potential. A related use of the $C^{1,1/3+}$ regularity threshold also appears in the work of Bhattacharya--Ogden \cite{Bhattacharya-Ogden}, which proved regularity for $C^{1,1/3+}$ solutions to the Hamiltonian stationary equation in the supercritical case.

 The proof in higher dimensions is almost the same. The phase condition $\Theta\in I_k$, together with semi-convexity, shows that $D^2u$ can have at most $k$ blow up eigenvalues. On the other hand, by the Hessian duality described above, the $W^{2,k}$ assumption again leads to a contradiction if the singular set has positive measure in the rotated coordinates.
 
 The semi-convexity assumption in Theorem \ref{THM: high dim reg} is used only to ensure the flatness of tangent cones after blowing up the minimal Lagrangian graph. In dimensions $n=3$ or $4$, every graphical special Lagrangian cone is flat, see \cite{Fischer-Colbrie, Han-Nadi-Yuan} for $n=3$ and \cite{NV-4D-sLag} for $n=4$, then general semi-convexity is sufficient. In higher dimensions $n\ge 5$, there is no such rigidity result in general, and Bhattacharya--Orriols--Skorobogatova \cite{B-O-S} recently constructed a nonflat 5-dimensional graphical special Lagrangian cone. Therefore, we need to impose stronger eigenvalue constraints to ensure the rigidity of special Lagrangian cones. The $\tan(\pi/6)$-convexity condition is sufficient for this purpose \cite{Warren-Yuan-Liouville, Ding-Qi}. One may also replace it by the phase-dependent semi-convexity imposed by Ogden and Yuan \cite{Ogden-Yuan}.

\bigskip
\noindent
{\bf Organization.} This paper is organized as follows. In Section \ref{SEC: Pre}, we introduce the Lewy-Yuan rotation and Hessian measures for general semi-convex functions. In Section \ref{SEC: minimality}, we prove the minimality of the gradient graph, namely Theorem \ref{THM: minimality in high dim}. Finally, in Section \ref{SEC: Pf of THMs}, we prove the regularity results in Theorem \ref{thm: 3D main} and Theorem \ref{THM: high dim reg}.

\bigskip
\noindent
{\bf Acknowledgments.}  
The author is grateful to Yu Yuan for bringing this problem to the author's attention, and for his patient guidance and constant encouragement. The author also thanks Ravi Shankar for helpful discussions.

\section{Preliminaries}\label{SEC: Pre}

\subsection{The Lewy-Yuan rotation}\label{Subsec: LY rotation}
For a convex function $u$ defined on $B_1$, the subgradient of $u$ at $x_0\in B_1$ is defined by the set of slopes of all supporting hyperplane of $u$ at $x_0$. At a differentiability point, the subgradient is a singleton, namely the gradient at that point. This notation extends naturally to semi-convex functions by making a shift. More precisely,  let $u$ be a semi-convex function on $B_1\subset \bb{R}^n$ with $u+ \tan\theta\, |x|^2/2$ is convex for some $\theta\in (0,\pi/2)$.  Then we define $\partial u(x_0):= \partial(u+\tan\theta\, |x|^2/2) (x_0) - \tan\theta\, x_0$. In this case, the gradient graph of $u$ is understood as $(x,\partial u(x))$.

The Lewy--Yuan rotation is a rotation of the gradient graph of a semi-convex function that makes the graph have bounded slope in the rotated coordinates. Yuan first introduced this rotation into the study of the special Lagrangian equations and used it to prove a number of rigidity and regularity results \cite{Yuan-Invent, Yuan-06, Wang-Yuan-singular, Chen-Warren-Yuan, Chen-Shankar-Yuan, Ogden-Yuan}. The rotation was first developed for smooth potentials. For lower-regularity potentials, Warren \cite{Warren} and Chen--Warren \cite{Chen-Warren} showed that the rotation can still be carried out for $C^1$ potentials. Later, Chen--Shankar--Yuan \cite{Chen-Shankar-Yuan} extended the rotation to general continuous semi-convex functions using the Legendre transform. We describe this rotation below.

Assume that $u+\tan\theta \,|x|^2/2$ is convex for some $\theta\in(0,\pi/2)$. For any angle $\beta \in (0, \pi/2-\theta)$, we can perform the downward rotation by angle $\beta$ on the gradient graph $(x,\partial u (x))$. Let $c=\cos \beta$ and $s=\sin\beta$. In the rotated coordinate $(\bar x, \bar y)= ( cx+sy, -sx+cy )$, the gradient graph $(x,\partial u(x))$ can be realized as a Lipschitz gradient graph $(\bar x, D\bar u(\bar x))$, where the new  potential is given by
\begin{equation}\label{eq: rotated potential}
    \bar u(\bar x)= \dfrac{c}{s}\dfrac{|\bar x|^2}{2} - \dfrac{1}{s} \left( su+c\dfrac{|x|^2}{2} \right)^*(\bar x),\quad \text{for}\ \bar x\in \partial (su+c|x|^2/2  )(B_1).
\end{equation}
Here the superscript $^*$ denotes the Legendre transform. Since $
U(x)=s u(x)+c|x|^2/2$ is uniformly convex, its Legendre transform is well defined by
\[ U^*(\bar x)=\sup_{x\in B_1}\bigl(\bar x\cdot x-U(x)\bigr),
\quad \bar x\in \partial U(B_1)=:\Omega . \]

We summarize below several properties of this new potential $\bar u$. The proofs can be found in \cite{Chen-Shankar-Yuan}.

(1) The potential $\bar u$ is a $C^{1,1}$ function on the domain $\Omega$ with the Hessian bound:
\[ -\tan(\theta+\beta) I \leq D^2\bar u\leq \tan\left( \dfrac{\pi}{2}-\beta \right). \]

(2) If $u$ is twice differentiable at the point $x_0\in B_1$, then $\bar u$ is also twice differentiable at the corresponding point $\bar x_0= DU(x_0)$. If $ \bar u$ is twice differentiable at the point $\bar x_0$ with $D^2 \bar u< \tan(\pi/2 - \beta)$, then $ u$ is also twice differentiable at $x_0$.
In either case, their Hessians have the relation:
\[ D^2 \bar u(\bar x_0) = (-sI + cD^2u(x_0)) (cI+sD^2u(x_0))^{-1}.  \]
In particular, let $\lambda_i$ and $\bar \lambda_i$ denote the eigenvalues of $D^2u(x_0)$ and $D^2\bar u(\bar x_0)$, respectively. Set $\theta_i=\arctan \lambda_i$ and $\bar\theta_i = \arctan \bar\lambda_i$. We also have
\[ \bar \lambda_i = \dfrac{-s+ c\lambda_i}{c+s\lambda_i}\qquad \text{and} \qquad  \bar\theta_i = \theta_i-\beta. \]

(3) (Preservation of supersolutions, \cite[Proposition 3.2]{Chen-Shankar-Yuan}) If $u$ is a viscosity supersolution of \eqref{eq: sLag} with phase $\Theta$ on $B_1$, then $\bar u$ is a viscosity supersolution of \eqref{eq: sLag} with phase $\Theta-n\beta$ on $\Omega$.

\subsection{Hessian measures}\label{Subsec: Hessian measures}
In this subsection, we introduce Hessian measures for convex functions, with the aim of extending the $\sigma_k$-Hessian operator $\sigma_k(D^2u)$ to functions with low regularity. Here, $\sigma_k(D^2u)$ denotes the $k$-th elementary symmetric function of the eigenvalues of $D^2u$.

There are several ways to define $k$-Hessian measures. For example, Trudinger and Wang \cite{TW-Hess-1, TW-Hess-2, TW-Hess-3} defined $k$-Hessian measures for general $k$-convex functions by smooth approximation. For our purposes, we use the definition introduced by Colesanti \cite{Colesanti--formula,Colesanti-Salani}, in which the Hessian measures arise as the coefficients in a Steiner-type formula for mixed volumes. This definition applies to both convex and semi-convex functions.  The two definitions are equivalent for convex and semi-convex functions. However, the latter is more geometrical and  better suited to our purposes, since it makes the duality of Hessian measures under the Legendre transform more transparent.
 
 Let $u$ be a convex function defined on $B_1$, and let $\Gamma_{u}=\{(x, p): x\in B_1, p\in\partial u(x)\}$ denote its gradient graph. For any Borel set $A\subset \Gamma_u$ and $t>0$, consider the parallel set
\[ P_{t}[u](A):=\{ x+tp: (x,p)\in A \}. \]
The Steiner formula for convex functions \cite[Theorem 1.1]{Colesanti--formula} asserts that there exist nonnegative Borel measures $\mu_{k}[u]$ on $\Gamma_u$ for $k=0,1,\dots, n$, such that
\begin{equation}\label{eq: Steiner}
    \mathcal{H}^n(P_{t}[u](A))= \sum_{k=0}^{n}t^{k}\mu_{k}[u](A).
\end{equation}
The measure $\mu_k[u]$ is called the $k$-Hessian measure of $u$. We also regard $\mu_{k}[u]$ as a measure on $B_1\times \bb{R}^n$ by setting  $\mu_{k}[u](A)= \mu_k[u](A\cap \Gamma_u)$ for any Borel set $A\subset B_1\times \bb{R}^n$. By \cite[Theorem 1.1]{Colesanti-Hug-MM}, these Hessian measures are weakly continuous under locally uniform convergence of convex functions.

When $u\in C^2$, the $k$-Hessian measures defined in this way recover the classical measures $\sigma_{k}(D^2u)\,\dd x$. Indeed, for any Borel set $E\subset B_1$,  the area formula gives
\begin{align*}
    \mathcal{H}^n(P_{t}[u](E\times\bb{R}^n)) = \int_{E} \det (I+tD^2u(x))\,\mathrm{d}x=\sum_{k=0}^{n} t^{k}\int_{E}\sigma_k(D^2u)\,\dd x.
\end{align*}
Comparing this with the Steiner formula \eqref{eq: Steiner} yields
\[
\mu_k[u](E\times \mathbb R^n)
=
\int_E \sigma_k(D^2u)\,\mathrm dx.
\]
By mollification, one can also show that the lower-order Hessian measures of a $W^{2,p}$ convex function also recover the corresponding classical measures.
\begin{proposition}
    Let $u\in W^{2,p}(B_1)$ be a convex function defined on $B_1$ for some $p\ge 1$. Then for $k=1,\dots, [p] $, we have
    \[
    \mu_k[u](E\times \mathbb R^n)
    =
    \int_E \sigma_k(D^2u)\,\mathrm dx\quad \text{for any Borel set}\ E\subset B_1
\]
\end{proposition}

Next, we study the duality of Hessian measures under the Legendre transform. We first explain the motivation in the smooth setting. Let $u$ be a smooth uniformly convex function with $D^2u\ge \delta I>0$. Geometrically, if we view the gradient graph $(x,Du(x))$ as a graph over the vertical coordinates, then the Legendre transform $u^*$ of $u$ is the corresponding potential of this gradient graph in those coordinates. Namely, $(x,Du(x))=(Du^*(y),y).$ The coordinate changes $y=Du(x)$ and $x=Du^*(y)$ are inverse to each other, and $D^2u(x)= (D^2u^*(y))^{-1}$. Let $\lambda_i$ and $\mu_i$ denote the eigenvalues of $D^2u(x)$ and $D^2u^*(y)$, respectively. Then $\mu_i=\lambda_i^{-1}$.
Using the change of variables $\dd x= \det D^2 u^*(y) \, \dd y$, we obtain the duality relation
\[ \sigma_k(D^2u(x))\,\dd x= \sigma_k\left(\dfrac{1}{\mu_1},\cdots,\dfrac{1}{\mu_n}\right) \det D^2u^*(y)\,\dd y= \sigma_{n-k}(D^2u^*(y))\,\dd y. \]
By Colesanti-Hug \cite[Theorem 5.8]{Colesanti-Hug-MM}, this duality extends to the $k$-Hessian measures of general convex functions.
\begin{proposition}\label{PROP: Hessian dual}
    Let $u$ be a convex function defined on $B_1\subset \bb{R}^n$, and let $u^*$ be its Legendre transform. Then for any Borel sets $\alpha\subset B_1$ and $\beta \subset \partial u(B_1)$, we have
    \[  \mu_{k}[u](\alpha\times \beta) = \mu_{n-k}[u^*](\beta \times \alpha),\quad \text{for}\ k=0,1,\cdots,n.  \]
\end{proposition}
\begin{proof}
    From the subgradient duality of the Legendre transform, we have $p\in\partial u(x) \Longleftrightarrow x\in \partial u^*(p) $. Therefore, for any $t>0$, we obtain
    \begin{align*}
        P_{t}[u](\alpha\times \beta)&= \{ x+tp: x\in \alpha, p\in\beta, p\in \partial u(x) \}\\
        &= t\{ p+ x/t: p\in\beta, x\in \alpha,x\in \partial u^*(p) \}\\
        &= tP_{1/t}[u^*](\beta\times \alpha).
    \end{align*}
    Applying the Steiner formula \eqref{eq: Steiner}, we compute the expansion of its $\mathcal{H}^n$-measure to obtain
    \begin{align*}
    \sum_{k=0}^{n}t^{k} \mu_{k}[u](\alpha\times \beta) =t^n\sum_{k=0}t^{-k} \mu_{k}[u^*](\beta\times\alpha).
    \end{align*}
    Comparing the coefficients of the $t^k$ terms on both sides gives us desired duality.
\end{proof}

\section{Minimality of the gradient graph}\label{SEC: minimality}

In this section, we prove that minimiality assertions in Theorem \ref{thm: 3D main} and \ref{THM: minimality in high dim}. It suffices to consider the subcritical phase case. We use the notation introduced in Section \ref{Subsec: LY rotation}. Recall that
\[ U(x)= su(x) + c\dfrac{|x|^2}{2} \qquad \text{and}\qquad U^*(\bar x)= -s\bar u(\bar x) + c\dfrac{|\bar x|^2}{2}. \]
Notice that $U^*\in C^{1,1}$, it follows that $U^*$ is twice differentiable almost everywhere. Denote
\[ \mathcal{R}=\{\bar x\in \Omega: \bar U \text{ is twice differentiable at }\bar x_0\text{ and } \det D^2\bar U(\bar x_0)>0. \} \]
and
\[ \mathcal{S}=\{\bar x\in \Omega: \bar U \text{ is twice differentiable at }\bar x_0\text{ and } \det D^2\bar U(\bar x_0)=0. \} \]
We call $\mathcal{R}$ the \emph{regular set} because, for $\bar x\in\mathcal{R}$, we have $\bar u$ is twice differentiable at $\bar x$ with  $D^2\bar{u}(\bar x)< c/s= \tan(\pi/2-\beta)$. Thus, the Property (2) of the Lewy--Yuan rotation allows us to directly verify the equation for $\bar u$ at this point:
\begin{align*}
    F(D^2\bar u (\bar x)) =\sum_{i=1}^{n} \arctan \bar\lambda_i = \sum_{i=1}^{n}\arctan\lambda_i-\beta=\Theta-n\beta.
\end{align*}
Similarly, we call $\mathcal{S}$ the \emph{singular set} because, at the corresponding points in the original $x$-coordinates, $D^2u$ blows up.
\begin{proposition}\label{PROP: partial reg}
    The preimage of $\mathcal{S}$ in the $x$-coordinate has measure zero, that is $|DU^*(\mathcal{S})|$=0.
\end{proposition}
\begin{proof}
    In the $x$-coordinate, $u$ is a semi-convex viscosity solution to \eqref{eq: sLag}. By the Alexandrov theorem, $u$ is twice differentiable almost everywhere in $B_1$. Let $\mathcal Z$ denote the set of such twice differentiability points, then $|B_1\setminus\mathcal{Z}|=0.$ Fix any $x_0\in \mathcal{Z}$, there exists a quadratic polynomial $Q_{x_0}(x)$ such that $|u(x)-Q_{x_0}(x)|=o(|x-x_0|^2)$ as $x\to x_0$. In particular, $F(D^2Q_{x_0})=\Theta$. For $r>0$ small, to be chosen later, consider the rescaled function 
    \[ v(y)= \dfrac{u(x_0+ry)-Q_{x_0}(x_0+ry)}{r^2}\qquad \text{for}\ y\in B_1, \]
    then $v$ solves the equation
    \[ G(D^2v)= F(D^2v+D^2Q_{x_0})-\Theta=0. \]
    This fully nonlinear operator $G(M)=F(M+D^2Q_{x_0})-\Theta$ satisfies the hypotheses of Savin's small perturbation theorem \cite{Savin}: $G$ is uniformly elliptic for $|M|\leq 1$, $G(0)=0$ and $|D^2G|\leq C(n)$. Let $\delta>0$ be the small constant in Savin's theorem, we have $\|v\|_{L^{\infty}(B_1)}\leq o(r^2)/r^2 <\delta$ provided $r$ is sufficiently small. Savin's small perturbation theorem then implies that $v\in C^{2,\alpha}(B_{1/2})$, and hence $u\in C^{2,\alpha}(B_{r/2}(x_0))$. Consequently, $u$ is smooth near $x_0$.  It is then straightforward to check that $DU(\mathcal{Z})\subset \mathcal{R}$. Equivalently, $\mathcal{Z}\subset DU^*(\mathcal{R})$, and hence $DU^*(\mathcal{S})\subset B_1\setminus \mathcal{Z}$. This completes the proof.
\end{proof}

The rest of this subsection is devoted to proving that $\mathcal{S}$ has measure zero in the rotated $\bar x$-coordinates.

\subsection{3D case}\label{Subsec: 3D minimality}
A key lemma is that, under the $W^{2,1}$ assumption on $U$, the Hessian $D^2U^*$ has at least two zero eigenvalues for almost every point in $\mathcal{S}$ when $\mathcal{S}$ has positive Lebesgue measure.
\begin{lemma}\label{LEMMA 3D 2 blowup eigen}
    Suppose that $|\mathcal{S}|>0$ and $U\in W^{2,1}(B_1)$, then we have
    \[ \mathrm{rank}\, D^2U^*\leq 1 \quad a.e.\ in\  \mathcal{S}.  \]
\end{lemma}
\begin{proof}
     Since $\det D^2U^*=0$ on $\mathcal{S}$, it remains to rule out the possibility that $D^2U^*$ has rank two. Let
     \[ \mathcal{S}_2:= \{ \bar x \in \mathcal{S}: \mathrm{rank}\, D^2U^*(\bar x)=2  \}. \]
     On $\mathcal{S}_2$, the matrix $D^2U^*$ has two positive eigenvalues and one zero eigenvalue. In particular, $\sigma_{2}(D^2U^*)>0$. Applying Proposition \ref{PROP: Hessian dual} to $U$ and $U^*$ with $\alpha=DU^*(\mathcal{S}_2)$ and $\beta = \mathcal{S}_2$, we obtain
     \begin{equation}\label{eq: 3D hess dual}
         \int_{DU^*(\mathcal{S}_2)}\Delta U\,\dd x=\mu_{1}[U](DU^*(\mathcal{S}_2)\times \mathcal{S}_2) = \mu_2[U^*](\mathcal{S}_2\times DU^*(\mathcal{S}_2))= \int_{\mathcal{S}_2}\sigma_2(D^2U^*) \,\dd \bar x 
     \end{equation}
     Here we have used the facts that $U\in W^{2,1}$ and $U^*\in C^{1,1}$, which ensure that their Hessian measures in \eqref{eq: 3D hess dual} reduce to their classical forms. The LHS vanishes because $DU^*(\mathcal{S}_2)$ has measure zero by Proposition \ref{PROP: partial reg}. Therefore, the positivity of $\sigma_2(D^2U^*)$ implies that $|\mathcal{S}_2|=0$. The proof is complete.
\end{proof}

Lemma \ref{LEMMA 3D 2 blowup eigen} shows that, if $|\mathcal{S}|>0$, then whenever $D^2u$ blows up, it must have at least two blow-up eigenvalues. However, at three dimensional subcritical phase, $D^2u$ is allowed to have at most one blow up eigenvalue. This gives a contradiction. The rigorous proof is more convenient in the rotated $\bar x$-coordinate.

\bigskip
\emph{Proof of the minimality in Theorem \ref{thm: 3D main}.} Assume that $|\mathcal{S}|>0$. Since $U^*\in C^{1,1}$, it is twice differentiable almost everywhere. Hence we can choose a point $\bar x_0\in\mathcal{S}$ at which $U^*$ is twice differentiable. By Lemma \ref{LEMMA 3D 2 blowup eigen}, we may also assume that $\mathrm{rank}\, D^2U^*(\bar x_0)\leq 1$. Thus $D^2U^*(\bar x_0)$ has at least two zero eigenvalues. Using the expression in  \eqref{eq: rotated potential}, we see that $D^2\bar u(\bar x_0)$ has at least two maximal eigenvalues $\tan(\pi/2-\beta)$. Together with the lower bound for $D^2\bar u$, this contradicts the fact that $\bar u$ is  a supersolution. Precisely,
\[ \Theta-3\beta \ge F(D^2\bar u(\bar x_0)) \ge  \left( \dfrac{\pi}{2} -\beta \right) + \left( \dfrac{\pi}{2} -\beta \right) + \arctan \bar \lambda_3 \ge \pi-\theta-3\beta. \]
This is impossible, since $\Theta , \theta <\pi/2$. Therefore, $|\mathcal{S}|=0$. It follow that $\bar u$ is a $C^{1,1}$ strong solution to $F(D^2\bar u)= \Theta-3\beta$ on $\Omega$. The area-minimizing property of the gradient graph then follows from the calibration argument \cite[Theorem 4.2]{Harvey-Lawson}. \hfill\qed

\subsection{General case}\label{Subsec: general minimality}
The proof of the higher dimensional case is similar.
\begin{lemma}\label{LEMMA: nD Hess dual}
    Suppose that $|\mathcal{S}|>0$ and $U\in W^{2,k}(B_1)$ for some $k\in \{ 1, 2,\dots, n-2\}$, then we have
    \[ \mathrm{rank}\, D^2U^*\leq n-k-1 \quad a.e.\ in\  \mathcal{S}.\]
\end{lemma}
\begin{proof}
    For each $l=n-k, n-k+1,\dots, n-1$, set
    \[  \mathcal{S}_{l}= \{ \bar x\in\mathcal{S}: \mathrm{rank} \, D^2U^*(\bar x) = l \}.  \]
    We show that $|\mathcal{S}_l|=0$ for each $l$.
    
    On $\mathcal{S}_l$, we have $\sigma_{l}(D^2U^*)>0$. Applying Proposition \ref{PROP: Hessian dual}, to $U$ and $U^*$ with $\alpha=DU^*(\mathcal{S}_l)$ and $\beta=\mathcal{S}_l$, we obtain
    \begin{equation*}
         \int_{DU^*(\mathcal{S}_l)}\sigma_{n-l}(D^2U)\,\dd x=\mu_{n-l}[U](DU^*(\mathcal{S}_l)\times \mathcal{S}_l) = \mu_l[U^*](\mathcal{S}_l\times DU^*(\mathcal{S}_l))= \int_{\mathcal{S}_l}\sigma_l(D^2U^*) \,\dd \bar x.
     \end{equation*}
     The LHS vanishes since $U\in W^{2,k}$ with $ k\ge n-l$ and $|DU^*(\mathcal{S}_l)|=0$, thus $|\mathcal{S}_l|=0$ follows.
\end{proof}

\bigskip
\emph{Proof of Theorem \ref{THM: minimality in high dim}.} Assume that $|\mathcal{S}|>0$.  By Lemma \ref{LEMMA: nD Hess dual}, we may choose a point $\bar x_0 \in \mathcal{S}$ such that $U^*$ is twice differentiable at $\bar x_0$ and $\mathrm{rank}\,D^2U^*(\bar x_0)\leq n-k-1$. Thus, $D^2\bar u(\bar x_0)$ has at least $k+1$ maximal eigenvalues $\tan(\pi/2-\beta)$. Since $\bar u$ is a supersolution, we obtain
\begin{equation}\label{eq: contradict}
    \Theta-n\beta \ge F(D^2\bar u(\bar x_0)) \ge  (k+1)\left( \dfrac{\pi}{2}-\beta\right) - (n-k-1)(\theta+\beta).
\end{equation}
Notice that $\Theta \in I_k $, then $\Theta\leq k\pi-(n-2)\pi/2$. Simplifying \eqref{eq: contradict}, we get
\[ (n-k-1)\theta \ge (k+1)\dfrac{\pi}{2} - \Theta \ge (n-k-1)\dfrac{\pi}{2}, \]
which is impossible. Therefore, $|\mathcal{S}|=0$, and the area-minimizing property of the gradient graph follows. \hfill\qed

\section{Proof of Theorem \ref{thm: 3D main}  and Theorem \ref{THM: high dim reg}} \label{SEC: Pf of THMs}
In this section, we finish the proofs of Theorem \ref{thm: 3D main} and Theorem \ref{THM: high dim reg}.  We have already known that the gradient graph $(x,\partial u(x))$ is area-minimizing. In fact, once minimality is established, the smoothness of the gradient graph follows from the semi-convexity condition in Theorem \ref{THM: minimality in high dim}. More precisely, $u+\tan\theta |x|^2/2 $ is convex for some fixed constant $\theta\in(0,\pi/2)$ when $n=3$ or $4$, while $u+\tan(\pi/6)|x|^2/2$ is convex when $n\geq 5$.

\subsection{Smoothness of the rotated potential} We rotate the gradient graph of  $u$ by the angle $\beta$, where
\[ \beta= \begin{cases}
    \text{any fixed angle in }(0,\pi/2-\theta),\quad &\text{when }n=3\ \text{or}\ 4; \\
    \pi/6 &\text{when }n\ge 5.
\end{cases} \]
Then we obtain a $C^{1,1}$ potential $\bar u$ which is a viscosity solution of the corresponding special Lagrangian equation in the rotated $\bar x$-coordinates. Moreover, in dimensions $n\ge 5$, we have $\|D^2\bar u\|_{L^{\infty}}\leq \sqrt{3}$.

\begin{proposition}
    Let $u$ be a $C^{1,1}$ viscosity solution to the special Lagrangian equation \eqref{eq: sLag} on a domain $\Omega\subset \bb{R}^n$. Then

    (1) $u$ is smooth in $\Omega$ when $n=3$ or $4$;

    (2) In general dimensions, if $\|D^2u\|_{L^{\infty}(\Omega)}\leq \sqrt{3}$, then $u$ is smooth.
\end{proposition}
\begin{proof}
    Denote $K= \|D^2u\|_{L^{\infty}}$. For any fixed $x_0\in\Omega$ with $B_{\rho}(x_0)\in \Omega$ , and any $k\in \bb{N}^{+} $, consider the blow up
    \[ u_{k}(x)= k^2 u \left(x_0 + \dfrac{x}{k} \right),\quad \text{for }x\in B_{k\rho}. \]
    For any $R>0$, we see that $\{u_k\}$ is a family of $C^{1,1}$ viscosity solutions to \eqref{eq: sLag}, with $\|u_k\|_{C^{1,1}(B_{R})}\leq C(K,R)$. Then, up to a subsequence, there is a function $u_R\in C^{1,1}(B_R)$ such that $u_k\to u_R$ in $C^{1,\alpha}(B_R)$ as $k\to\infty$ for any $\alpha\in (0,1)$ and $|D^2u_R|\leq K$. Since the uniformly convergence preserves viscosity solutions, we know that $u_R$ is also a viscosity solution to \eqref{eq: sLag}. Now the \eqref{eq: sLag} is uniformly elliptic, by the $W^{2,\delta}$ estimate \cite[Proposition 7.4]{CC-book}, we have  
    \[  \| D^2u_k - D^2u_R \|_{L^{\delta}(B_{R/2})} \leq C(K,R)\| u_k-u_{R} \|_{L^{\infty}(B_{R})} \to 0 ,\quad \text{as}\ k\to \infty.  \]
    Note that $|D^2u_k|,|D^2u_{R}|\leq K$, then for any $p>0$, we have
    \[ \| D^2u_k - D^2u_R \|_{L^{p}(B_{R/2})}\to 0 ,\quad \text{as}\ k\to \infty. \]
    By the standard diagonalizing process, up to a further subsequence, there exists $u_{\infty}\in C^{1,1}(\bb{R}^n)$ such that $u_k\to u_{\infty}\in W^{2,p}_{loc}(\bb{R}^n)$ as $k\to\infty$, $|D^2u_{\infty}|\leq K$, and $u_{\infty}$ is a viscosity solution to \eqref{eq: sLag} on $\bb{R}^n$. Since $(x, Du(x))$ is a minimal submanifold in $\bb{R}^n\times \bb{R}^n$, by the monotinicity formula, $(x,Du_{\infty}(x))$ is a cone, that is $u_{\infty}$ is homogeneous of order 2. When the dimension $n=3$ or $4$, $u_{\infty}$ is a quadratic polynomial by the rigidity of graphical special Lagrangian cones, by \cite{Fischer-Colbrie, Han-Nadi-Yuan,NV-4D-sLag}. If $|D^2u|\leq K=\sqrt{3}$, the eigenvalue constraint \cite{Warren-Yuan-Liouville, Ding-Qi} holds, we also conclude that $u_{\infty}$ is a quadratic polynomial. 

    For sufficiently large $k$, we know that each $v_k=u_k-u_\infty$ is a viscosity solution to
    \[ G(D^2v_k)= F(D^2v_k + D^2u_{\infty}) - F(D^2u_{\infty}) =0 \quad \text{on}\ B_1.  \]
    Similar to the proof of Proposition \ref{PROP: partial reg}, the operator $G$ satisfies the hypotheses of Savin's small perturbation theorem \cite{Savin}. Let $\delta>0$ be the universal constant in Savin's theorem, we have $\|v_k\|_{L^{\infty}(B_1)} = \|u_k-u_{\infty}\|_{L^{\infty}(B_1)} \leq \delta$ for sufficiently large $k$. Therefore $v_k\in C^{2,\alpha}(B_{1/2})$ and hence smooth in $B_{1/2}$. This implies that $u$ is smooth in $B_{1/2k}(x_0)$. Since $x_0$ is arbitrary, the smoothness of $u$ follows. 
\end{proof}

From this proposition, the rotated potential $\bar u$ is smooth, and also implies the smoothness of the gradient graph.

\subsection{Regularity of the original potential}
The following Lemma, due to Bhattacharya--Shankar \cite[Theorem 4.1]{AB-RS-ARMA}, allows us to pass the smoothness of the rotated potential back to the original $C^{1,1/3+}$ potential.

\begin{lemma}\label{LEMMA: BS lemma}
    Let $U$ be a $C^{1,\alpha}$ strictly convex function defined on $B_1$ for some $\alpha>1/3$. If its Legendre transform $U^*$ is smooth in $DU(B_1)$, then $D^2U^*$ is nondegenerate, that is, $\det D^2U^*>0$ in $DU(B_1)$.
\end{lemma}
\begin{proof}
    Suppose, to the contrary, that $\det D^2U^*=0$ at some point $\bar x_0\in DU(B_1)$. Then there exists a direction $e\in \bb{S}^{n-1}$ such that $e^{T}D^2U^*(\bar x_0) e=0$. Consider the smooth function $g(\bar x)= e^{T}D^2U^*(\bar x) e$ near $\bar x_0$.  Since $U^*$ is convex, we have $g\ge0$. Thus $g$ attains its local minimum at $\bar x_0$, it follows that $Dg(\bar x_0)=0$. This yields 
    \[ D^3U^*(\bar x_0)[e,e,w]=0 \quad \text{for any direction }w\in\mathbb{S}^{n-1}. \]
    Now, consider the expansion of $DU^*$ along the line $\bar x_0+te$, i.e., consider the map $h(t)=DU^*(\bar x_0+te)-DU^*(\bar x_0)$, then $h$ is smooth and $h(0)= h'(0)=h''(0)=0$. By Taylor's expansion of $h$, we obtain $h(t)=O(t^3)$. 

    On the other hand, $DU^*$ has  inverse $DU\in C^{\alpha}$ for some $\alpha>1/3$, we conclude that
    \begin{align*}
        |t|= | DU(DU^*(\bar x_0+te)) - DU(DU^*(\bar x_0))|\leq C|h(t)|^\alpha \leq C|t|^{3\alpha}.
    \end{align*}
    Since $3\alpha>1$, this gives a contradiction by sending $t\to 0$.
\end{proof}

\bigskip
\noindent\emph{Proof of the remaining part of Theorem \ref{thm: 3D main} and Theorem \ref{THM: high dim reg}.} Now assume that $u\in C^{1,\alpha}$ for some $\alpha>1/3$. Applying Lemma \ref{LEMMA: BS lemma} to
\[ U(x)=su(x)+c\dfrac{|x|^2}{2}\in C^{1,\alpha}\qquad \text{and} \qquad U^*(\bar x)=-s\bar u(\bar x) + c\dfrac{|\bar x|^2}{2} \in C^{\infty}, \]
we obtain that $D^2U^*$ is nondegenerate in $DU(B_1)$. For any $x_0\in B_1$, at the corresponding point $\bar x_0$ in the rotated coordinate, we have $\det D^2U^*(\bar x_0) >0 $. By the inverse function theorem, $DU$, as the inverse of $DU^*$, is smooth near $x_0$. This implies the smoothness of the original potential $u$. \hfill\qed

\bigskip
\noindent
{\bf Declaration on the use of AI.} During the preparation of this manuscript, the author used ChatGPT 5.6 Sol to help make parts of the proofs more rigorous. In particular, ChatGPT pointed the author to a series of works by Colesanti and his collaborators \cite{Colesanti--formula, Colesanti-Salani, Colesanti-Hug-MM} that were previously unknown to the author, and these references provided a rigorous framework in which some of the ideas in the paper could be formulated. The mathematical problem and the overall proof strategy were developed by the author. ChatGPT was also used for language polishing to improve the readability of the manuscript. The author has verified all mathematical statements and arguments and takes full responsibility for the manuscript.

\bibliographystyle{amsalpha}
\bibliography{ref}
\end{document}